\documentclass[12pt]{article}
\usepackage{amssymb}
\usepackage{amsthm}
\usepackage{amsmath}
\usepackage{graphicx}
\usepackage{enumerate}

\DeclareMathOperator{\Con}{Con}

\DeclareMathOperator{\BCon}{\mathbf{Con}}

\newtheorem{theorem}{Theorem}
\newtheorem{definition}[theorem]{Definition}
\newtheorem{lemma}[theorem]{Lemma}
\newtheorem{proposition}[theorem]{Proposition}
\newtheorem{remark}[theorem]{Remark}
\newtheorem{example}[theorem]{Example}
\newtheorem{corollary}[theorem]{Corollary}
\title{Remarks on special congruences}
\author{Ivan~Chajda and Helmut~L\"anger}
\date{}
\begin{document}

\footnotetext{Support of the research of both authors by the Austrian Science Fund (FWF), project I~4579-N, and the Czech Science Foundation (GA\v CR), project 20-09869L, entitled ``The many facets of orthomodularity'', is gratefully acknowledged.}

\maketitle

\begin{abstract}
We study algebras and varieties where every non-trivial congruence has some class being a non-trivial subuniverse of the algebra in question. Then we focus on algebras where this non-trivial class is a unique non-singleton class of the congruence. In particular, we investigate Rees algebras, quasi-Rees algebras and algebras having the one-block-property. We also present results concerning these properties on quotient algebras. Several examples of such algebras and varieties are included.
\end{abstract}

{\bf AMS Subject Classification:} 08A30, 08A40, 08B05, 08B20

{\bf Keywords:} congruence class, subuniverse, Rees algebra, Rees variety, quasi-Rees algebra, one-block-property

Congruences having the property that one of its classes is a subuniverse were studied by B.~Cs\'ak\'any \cite{C75} and R.~F.~Tichy \cite T who introduced the concept of Rees congruences, see also \cite{CD}. This concept was generalized for algebras and varieties satisfying the so-called one-block-property in \cite{AK}. The aim of the present paper is to modify these concepts and prove some results also for single algebras.

Let $\mathbf A=(A,F)$ be an algebra. The {\em algebra} $\mathbf A$ is called {\em non-trivial} if $|A|>1$. Denote by $\BCon\mathbf A=(\Con\mathbf A,\subseteq)$ the congruence lattice of $\mathbf A$ and by $\omega_A$ the least congruence on $\mathbf A$. A {\em congruence} $\Theta$ on $\mathbf A$ is called {\em non-trivial} if $\Theta\neq\omega_A$.

B.~Cs\'ak\'any \cite{C75} characterized those varieties $\mathcal V$ having the property that every congruence on an algebra $\mathbf A$ of $\mathcal V$ has some class being a subuniverse of $\mathbf A$. His result is as follows (see also Theorem~4.2.1 in \cite{CEL}).

\begin{proposition}\label{prop1}
For a variety $\mathcal V$ the following are equivalent:
\begin{enumerate}[{\rm(i)}]
\item Each congruence on a non-trivial member of $\mathcal V$ has some class being a subuniverse.
\item There exists some {\rm(}at most{\rm)} unary term $v(x)$ such that $\mathcal V$ satisfies the identity
\[
f\big(v(x),\ldots,v(x)\big)\approx v(x)
\]
for all fundamental operations $f$.
\end{enumerate}
\end{proposition}

The best known examples of varieties satisfying the conditions of Proposition~\ref{prop1} are the variety of groups where one can define $v$ as the nullary operation, i.e.\ the neutral element of the group, or the variety of idempotent algebras (e.g.\ lattices or semilattices) where one can take $v(x):=x$. Recall that an {\em algebra} is called {\em idempotent} if it satisfies the identity $f(x,\ldots,x)\approx x$ for every fundamental operation $f$ and that a {\em variety} is called {\em idempotent} if every of its members has this property. Hence, if $\mathcal V$ is an idempotent variety then it satisfies (ii) of Proposition~\ref{prop1} for $v(x)=x$.

However, it may happen that there exists some non-trivial algebra $\mathbf A$ and some non-trivial congruence $\Theta$ on $\mathbf A$ such that the only class of $\Theta$ which is a subuniverse of $\mathbf A$ is a singleton.

This cannot happen if the variety is congruence uniform. Recall that an {\rm algebra} $\mathbf A$ is called {\em congruence uniform} if for each congruence $\Theta$ on $\mathbf A$ all classes of $\Theta$ have the same cardinality. A {\em variety} $\mathcal V$ is called {\em congruence uniform} if each of its members has this property. The best known examples of such varieties are the variety of groups, the variety of rings and the variety of Boolean algebras.

An immediate consequence of Proposition~\ref{prop1} is the following result.

\begin{corollary}\label{cor2}
For a congruence uniform variety $\mathcal V$ the following are equivalent:
\begin{enumerate}[{\rm(i)}]
\item Each congruence on a non-trivial member of $\mathcal V$ has some class being a subuniverse which is non-trivial if the congruence is non-trivial.
\item There exists some {\rm(}at most{\rm)} unary term $v(x)$ such that $\mathcal V$ satisfies the identity
\[
f\big(v(x),\ldots,v(x)\big)\approx v(x)
\]
for all fundamental operations $f$.
\end{enumerate}
\end{corollary}

Corollary~\ref{cor2} applies to the variety of groups and the variety of rings.

Of course, the assumption of congruence uniformity is very strong and hence we try to find a weaker condition ensuring a condition similar to (i) of Corollary~\ref{cor2}.

Let $\mathcal V$ be a variety satisfying (i) of Corollary~\ref{cor2} and $F_\mathcal V(x,y)$ its free algebra with two free generators $x$ and $y$ and consider the congruence $\Theta(x,y)$ on $F_\mathcal V(x,y)$. Since $\Theta(x,y)$ is non-trivial it has a non-trivial class $C$. Then $C$ is of the form $[p_0(x,y)]\Theta(x,y)$ for some binary term $p_0(x,y)$. Hence there exist binary terms $p_1(x,y),\ldots,p_n(x,y)\in C\setminus\{p_0(x,y)\}$. Let $g$ denote the endomorphism of $F_\mathcal V(x,y)$ satisfying $g(x)=g(y)=x$. Then $g\big(t(x,y)\big)=t(x,x)$ for every $t(x,y)\in F_\mathcal V(x,y)$. Moreover, $(x,y)\in\ker g$ whence $\Theta(x,y)\subseteq\ker g$. Since $\big(p_0(x,y),p_i(x,y)\big)\in\Theta(x,y)$ for all $i=0,\ldots,n$ we have
\[
p_0(x,x)=g\big(p_0(x,y)\big)=g\big(p_i(x,y)\big)=p_i(x,x)
\]
for all $i=0,\ldots,n$. Although the following condition is neither a direct consequence of the previous nor a necessary one, we can consider it as a reasonable condition for such a variety $\mathcal V$. This condition says that
\begin{enumerate}
\item[(1)] there exists some positive integer $n$ and binary terms $p_0,\ldots,p_n$ such that
\[
p_0(x,y)=\cdots=p_n(x,y)\text{ if and only if }x=y.
\]
\end{enumerate}

The following lemma is almost evident.

\begin{lemma}
Let $\mathbf A=(A,F)$ be an algebra, $a,b\in A$ and $p$ a binary term and assume $\mathbf A$ to satisfy the identity
\[
f\big(p(x,x),\ldots,p(x,x)\big)\approx p(x,x)
\]
for every fundamental operation $f$. Then $[p(a,b)]\Theta(a,b)$ is a subuniverse of $\mathbf A$.
\end{lemma}

\begin{proof}
If $f$ is an $m$-ary fundamental operation and $a_1,\ldots,a_m\in[p(a,b)]\Theta(a,b)$ then
\begin{align*}
f(a_1,\ldots,a_m)\in[f\big(p(a,b),\ldots,p(a,b)\big)]\Theta(a,b) & =[f\big(p(a,a),\ldots,p(a,a)\big)]\Theta(a,b)= \\
& =[p(a,a)]\Theta(a,b)=[p(a,b)]\Theta(a,b).
\end{align*}
\end{proof}

We now return to varieties with congruences having a class being a non-trivial subuniverse.

\begin{theorem}
Let $\mathcal V$ be a variety where every non-trivial congruence on a member of $\mathcal V$ has some class being a non-trivial subuniverse. Then the following hold:
\begin{enumerate}
\item[{\rm(2)}] There exists a binary term $p_0$ such that $\mathcal V$ satisfies the identity
\[
f\big(p_0(x,x),\ldots,p_0(x,x)\big)\approx p_0(x,x)
\]
for every fundamental operation $f$.
\item[{\rm(3)}] The class $[p_0(x,y)]\Theta(x,y)$ is a non-trivial subuniverse of $F_\mathcal V(x,y)$.
\end{enumerate}
\end{theorem}

\begin{proof}
Consider the congruence $\Theta(x,y)$ on $F_\mathcal V(x,y)$. Because $\Theta(x,y)$ is non-trivial, there exists some $p_0\in F_\mathcal V(x,y)$ such that $[p_0]\Theta(x,y)$ is a non-trivial subuniverse of $F_\mathcal V(x,y)$. Since $F_\mathcal V(x,y)$ is the free algebra of $\mathcal V$ with two free generators $x$ and $y$, we have $p_0=p_0(x,y)$ for some binary term $p_0(x,y)$. Let $g$ denote the endomorphism of $F_\mathcal V(x,y)$ satisfying $g(x)=g(y)=x$. Then $g\big(t(x,y)\big)=t(x,x)$ for every $t(x,y)\in F_\mathcal V(x,y)$. Moreover, $(x,y)\in\ker g$ whence $\Theta(x,y)\subseteq\ker g$. Since $[p_0(x,y)]\Theta(x,y)$ is a subuniverse of $F_\mathcal V(x,y)$ we have $f\big(p_0(x,y),\ldots,p_0(x,y)\big)\mathrel{\Theta(x,y)}p_0(x,y)$ for all fundamental operations $f$ and hence
\[
f\big(p_0(x,x),\ldots,p_0(x,x)\big)=g\Big(f\big(p_0(x,y),\ldots,p_0(x,y)\big)\Big)=g\big(p_0(x,y)\big)=p_0(x,x)
\]
for all fundamental operations $f$.
\end{proof}

In the following theorem we show that the conditions (1) and (2) are sufficient for the fact that every non-trivial congruence on a member of $\mathcal V$ has some class being a non-trivial subuniverse.

\begin{theorem}\label{th1}
Let $\mathcal V$ be a variety such that there exists some positive integer $n$ and binary terms $p_0,\ldots,p_n$ satisfying the following two conditions:
\begin{enumerate}[{\rm(i)}]
\item $p_0(x,y)=\cdots=p_n(x,y)$ if and only if $x=y$.
\item $\mathcal V$ satisfies the identity
\[
f\big(p_0(x,x),\ldots,p_0(x,x)\big)\approx p_0(x,x)
\]
for every fundamental operation $f$.
\end{enumerate}
Then for every $\mathbf A=(A,F)\in\mathcal V$ and every non-trivial congruence $\Theta$ on $\mathbf A$ there exists some $a\in A$ such that $[p_0(a,a)]\Theta$ is a non-trivial subuniverse of $\mathbf A$.
\end{theorem}

\begin{proof}
Let $\mathbf A=(A,F)\in\mathcal V$ and $\Theta$ be a non-trivial congruence on $\mathbf A$. Then there exists some $a\in A$ such that $[a]\Theta$ is a non-trivial class of $\Theta$. Put $C:=[p_0(a,a)]\Theta$. If $c_1,\ldots,c_m\in C$ and $f$ is an $m$-ary fundamental operation then, according to (ii),
\[
f(c_1,\ldots,c_m)\in[f\big(p_0(a,a),\ldots,p_0(a,a)\big)]\Theta=[p_0(a,a)]\Theta=C
\]
showing $C$ to be a subuniverse of $\mathbf A$. Let $b\in[a]\Theta$ with $b\neq a$. Since $p_0(a,a)=\cdots=p_n(a,a)$ according to (i) we have $p_i(a,b)\in[p_i(a,a)]\Theta=[p_0(a,a)]\Theta=C$ for $i=0,\ldots,n$. According to Theorem~3.4 in \cite{DMS} condition (i) implies that there exists some positive integer $k$, ternary terms $t_1,\ldots,t_k$ and $u_1,\ldots,u_k,v_1,\ldots,v_k\in\{p_0,\ldots,p_n\}$ such that
\begin{align*}
                        x & \approx t_1\big(u_1(x,y),x,y\big), \\
t_i\big(v_i(x,y),x,y\big) & \approx t_{i+1}\big(u_{i+1}(x,y),x,y\big)\text{ for }i=1,\ldots,k-1, \\
t_k\big(v_k(x,y),x,y\big) & \approx y.
\end{align*}
Assume now $C$ to be a singleton. Then $p_i(a,b)=p_j(a,b)$ for all $i,j=0,\ldots,n$ and hence $u_i(a,b)=v_j(a,b)$ for all $i,j=1,\ldots,k$ and we obtain
\[
a=t_1\big(u_1(a,b),a,b\big)=t_1\big(v_1(a,b),a,b\big)=t_2\big(u_2(a,b),a,b\big)=\cdots=t_k\big(v_k(a,b),a,b\big)=b,
\]
a contradiction. Hence $C$ is a non-trivial subuniverse of $\mathbf A$.
\end{proof}
 
Recall that an {\rm algebra} $\mathbf A$ with an equationally definable constant $1$ is called {\em weakly regular with respect to $1$} if for all $\Theta,\Phi\in\Con\mathbf A$, $[1]\Theta=[1]\Phi$ implies $\Theta=\Phi$. A {\em variety} is called {\em weakly regular with respect to $1$} if any of its members has this property.

We recall the following characterization of varieties being weakly regular with respect to $1$ by B.~Cs\'ak\'any \cite{C70}, see also Theorem~6.4.3 in \cite{CEL}.

\begin{proposition}\label{prop2}
A variety with an equationally definable constant $1$ is weakly regular with respect to $1$ if and only if there exists some positive integer $n$ and binary terms $t_1,\ldots,t_n$ such that
\[
t_1(x,y)=\cdots=t_n(x,y)=1\text{ if and only if }x=y.
\]
\end{proposition}

\begin{corollary}\label{cor3}
Let $\mathcal V$ be a weakly regular variety with respect to the equationally definable constant $1$ satisfying the identity $f(1,\ldots,1)\approx1$ for every fundamental operation $f$. Then every non-trivial congruence on a member of $\mathcal V$ has some class being a non-trivial subuniverse.
\end{corollary}

\begin{proof}
According to Proposition~\ref{prop2} there exists some positive integer $n$ and binary terms $t_1,\ldots,t_n$ such that
\[
t_1(x,y)=\cdots=t_n(x,y)=1\text{ if and only if }x=y.
\]
If we take $p_0:=1$ and $p_i:=t_i$ for $i=1,\ldots,n$ then the assumptions of Theorem~\ref{th1} are satisfied.
\end{proof}

\begin{example}
Recall that a {\em loop} is an algebra $(L,\cdot,/,\backslash,1)$ of type $(2,2,2,0)$ satisfying the following identities:
\[
(x/y)y\approx x, (xy)/y\approx x, x(x\backslash y)\approx y, x\backslash(xy) \approx y, x1\approx1x\approx x.
\]
Every non-trivial congruence on a loop has some class being a non-trivial subuniverse. This can be seen as follows. Put
\begin{align*}
       n & :=1, \\
t_1(x,y) & :=x/y.
\end{align*}
If $t_1(x,y)=1$ then $x/y=1$ and hence $x=(x/y)y=1y=y$. If, conversely, $x=y$ then $t_1(x,y)=x/x=(1x)/x=1$. Moreover, $1\cdot1=1$, $1/1=(1\cdot1)/1=1$ and $1\backslash1=1\backslash(1\cdot1)=1$. Now apply Corollary~\ref{cor3}.
\end{example}

\begin{example}
An {\em implication algebra} is a groupoid $(I,\cdot)$ satisfying the identities
\[
(xy)x\approx x, (xy)y\approx(yx)x\text{ and }x(yz)\approx y(xz).
\]
It is well-known that the identity $xx\approx yy$ holds in every implication algebra. Hence this element is an equationally definable constant denoted by $1$. Further, the binary relation $\leq$ on $I$ defined by $x\leq y$ if and only if $xy=1$ {\rm(}$x,y\in I${\rm)} is a partial order relation on $I$, see e.g.\ Section~2.2 in {\rm\cite{CEL}}. Every non-trivial congruence on an implication algebra has some class being a non-trivial subuniverse. This can be seen as follows. Put
\begin{align*}
       n & :=2, \\
t_1(x,y) & :=xy, \\
t_2(x,y) & :=yx.
\end{align*}
Then $t_1(x,y)=t_2(x,y)=1$ if and only if $x=y$. Moreover, $1\cdot1=1$. Now apply Corollary~\ref{cor3}.
\end{example}

\begin{remark}\label{rem1}
Let $\mathbf A=(A,F)$ be an idempotent algebra. Then every congruence class of $\mathbf A$ is a subuniverse of $\mathbf A$ and hence every non-trivial congruence on $\mathbf A$ has some class being a non-trivial subuniverse.
\end{remark}

In particular, we have: Every non-trivial congruence on a lattice {\rm(}semilattice{\rm)} has some class being a non-trivial sublattice {\rm(}subsemilattice{\rm)}.

If the term $p_0(x,y)$ of {\rm(ii)} in Theorem~\ref{th1} is a constant term {\rm(}as in Corollary~\ref{cor3}{\rm)} then, by Theorem~3.9 in {\rm\cite{DMS}}, the variety $\mathcal V$ is congruence modular and $n$-permutable {\rm(}for some $n\geq2${\rm)}. This is not true if $p_0(x,y)$ is a non-constant term as in the variety of semilattices where we can put $n:=1$, $p_0(x,y):=x$ and $p_1(x,y):=y$ in order to satisfy the assumptions of Theorem~\ref{th1}. Recall that a {\em class} $\mathcal K$ of algebras of the same type is called {\em $n$-permutable} if for all $\mathbf A\in\mathcal K$ and all $\Theta,\Phi\in\Con\mathbf A$ we have
\[
\Theta\circ\Phi\circ\Theta\circ\Phi\circ\cdots=\Phi\circ\Theta\circ\Phi\circ\Theta\circ\cdots
\]
where on both sides of this equality there are $n$ congruences.

The concept of a Rees congruence was introduced in \cite T, but firstly used for semigroups by D.~Rees in 1940 under a different name. A {\em Rees congruence} on an algebra $\mathbf A=(A,F)$ is a congruence of the form $B^2\cup\omega_A$ where $B$ is a subuniverse of $\mathbf A$. A {\em Rees algebra} is an algebra $\mathbf A=(A,F)$ such that $B^2\cup\omega_A\in\Con\mathbf A$ for every subuniverse $B$ of $\mathbf A$. A {\em Rees variety} is a variety consisting of Rees algebras only.

It was shown in \cite{CEL} that every subalgebra and every homomorphic image of a Rees algebra is a Rees algebra again. Moreover, the following is proved (Theorem~12.2.6).

\begin{proposition}
For an algebra $\mathbf A=(A,F)$ the following conditions are equivalent:
\begin{enumerate}[{\rm(i)}]
\item $\mathbf A$ is a Rees algebra.
\item Every subalgebra of $\mathbf A$ generated by two elements is a Rees algebra.
\item $\big(p(a),p(b)\big)\in\langle\{a,b\}\rangle^2\cup\omega_A$ for any $a,b\in A$ and every $p\in P_1(\mathbf A)$.
\end{enumerate}
\end{proposition}

Here $\langle\{a,b\}\rangle$ denotes the subalgebra of $\mathbf A$ generated by $\{a,b\}$.

Results on Rees algebras are collected in Chapter~12 of \cite{CEL} which contains also the following characterization of Rees varieties (Theorem 12.2.7).

\begin{proposition}
A variety is a Rees variety if and only if it is at most unary.
\end{proposition}

Recall that a {\em variety} $\mathcal V$ is called {\em at most unary} if every proper term of $\mathcal V$ is essentially unary or nullary.

The concept of a Rees algebra is rather strong. Hence we modify it as follows.

\begin{definition}
An {\em algebra} $\mathbf A=(A,F)$ is called a {\em quasi-Rees algebra} if for every non-trivial congruence $\Theta$ on $\mathbf A$ there exists some class $C$ of $\Theta$ being a non-trivial subuniverse of $\mathbf A$ such that $C^2\cup\omega_A\in\Con\mathbf A$.
\end{definition}

The following is an immediate consequence of the definition of a Rees algebra and a quasi-Rees algebra, respectively.

\begin{lemma}
Let $\mathbf A=(A,F)$ be an idempotent Rees algebra. Then $\mathbf A$ is a quasi-Rees algebra.
\end{lemma}

\begin{proof}
Let $\Theta$ be a non-trivial congruence on $\mathbf A$. Then there exists some non-trivial class $C$ of $\Theta$. According to Remark~\ref{rem1}, $C$ is a subuniverse of $\mathbf A$ and, since $\mathbf A$ is a Rees algebra, $C^2\cup\omega_A\in\Con\mathbf A$.
\end{proof}

We are going to find a large class of quasi-Rees algebras.

Recall from \cite{CL} that a {\em{\rm(}join-{\rm)}directoid} is a groupoid $(D,\sqcup)$ satisfying the identities
\[
x\sqcup x\approx x, (x\sqcup y)\sqcup x\approx x\sqcup y, y\sqcup(x\sqcup y)\approx x\sqcup y\text{ and }x\sqcup\big((x\sqcup y)\sqcup z\big)\approx(x\sqcup y)\sqcup z.
\]
Hence the class of (join-)directoids forms a variety. In particular, every (join-)semilattice is a (join-)directoid. It is well-known that the binary operation $\leq$ on $D$ defined by $x\leq y$ if $x\sqcup y=y$ is a partial order relation on $D$, called the {\em order induced} by $\mathbf D$, and that $x,y\leq x\sqcup y$.

We can prove the following result.

\begin{theorem}\label{th2}
Let $\mathbf D=(D,\sqcup)$ be a finite join-directoid. Then $\mathbf D$ is a quasi-Rees algebra.
\end{theorem}

\begin{proof}
Let $\Theta$ be a non-trivial congruence on $\mathbf D$. Denote by $\leq$ the order induced by $\mathbf D/\Theta$. According to Remark~\ref{rem1} there exists some class of $\Theta$ being a non-trivial subuniverse of $\mathbf D$. Let $C$ be a maximal (with respect to $\leq$) class with this property. Since $D$ is finite, such a class $C$ exists. Moreover, for each $E\in D/\Theta$ with $C<E$ we have $|E|=1$. (Clearly, $E$ is a subuniverse of $\mathbf D$ for every $E\in D/\Theta$.) Put $\Phi:=C^2\cup\omega_S$. Of course, $\Phi$ is an equivalence relation on $D$ and $\Phi\subseteq\Theta$. Suppose $(a,b),(c,d)\in\Phi$. \\
Case~1. $a=b$ and $c=d$. \\
Then clearly $(a\sqcup c,b\sqcup d)\in\Phi$. \\
Case~2. $a\neq b$ and $c\neq d$. \\
Then $a,b,c,d\in C$. But $C$ is a subuniverse of $\mathbf D$. Thus also $a\sqcup c,b\sqcup d\in C$ and hence $(a\sqcup c,b\sqcup d)\in\Phi$. \\
Case~3. $a\neq b$ and $c=d$. \\
Then $C=[a]\Theta=[b]\Theta\leq[a\sqcup c]\Theta=[b\sqcup d]\Theta$. By the assumption on $C$, either $a\sqcup c,b\sqcup d\in C$ or $a\sqcup c=b\sqcup d$.  Hence $(a\sqcup c,b\sqcup d)\in\Phi$. \\
Case~4. $a=b$ and $c\neq d$. \\
This case is symmetric to Case~3. \\
Altogether, $\Phi\in\Con\mathbf S$, i.e.\ $\mathbf D$ is a quasi-Rees algebra.
\end{proof}

\begin{corollary}
According to Remark~\ref{rem1} and Theorem~\ref{th2} the variety $\mathcal D$ of {\rm(}join-{\rm)}di\-rec\-toids and, in particular, the variety $\mathcal S$ of {\rm(}join-{\rm)}semilattices has the following property: For each $\mathbf D=(D,\sqcup)\in\mathcal V$ and each non-trivial congruence $\Theta$ on $\mathbf D$ there exists some class of $\Theta$ being a non-trivial subdirectoid of $\mathbf D$ and if $D$ is finite then there exists such a class $C$ of $\Theta$ such that $C^2\cup\omega_D$ is a congruence on $\mathbf D$. For each $\mathbf S=(S,\sqcup)\in\mathcal S$ and each non-trivial congruence $\Theta$ on $\mathbf S$ there exists some class of $\Theta$ being a non-trivial subsemilattice of $\mathbf S$ and if $S$ is finite then there exists such a class $C$ such that $C^2\cup\omega_S$ is a congruence on $\mathbf S$.
\end{corollary}

Congruences having exactly one class that is not a singleton were treated in \cite{AK}.

\begin{definition}\label{def1}
An algebra $\mathbf A$ has the {\em one-block-property} if every atom of $\BCon\mathbf A$ has exactly one class which is not a singleton.
\end{definition}

Recall that a congruence $\Theta$ on $\mathbf A=(A,F)$ is called an {\em atom} of $\BCon\mathbf A$ if $\omega_A\prec\Theta$.

It was shown in \cite{AK} (see also Theorem~12.1.7 in \cite{CEL}) that a variety satisfying the one-block-property is congruence semimodular, but it need not be congruence modular.

The following result can be easily checked.

\begin{proposition}
Let $\mathbf A=(A,F)$ be a quasi-Rees algebra. Then it has the one-block-property.
\end{proposition}

\begin{proof}
Let $\Theta$ be an atom of $\BCon\mathbf A$. Since $\mathbf A$ is a quasi-Rees algebra, there exists some class $C$ of $\Theta$ being a non-trivial subuniverse of $\mathbf A$ such that $\Phi:=C^2\cup\omega_A\in\Con\mathbf A$. Clearly, $\Phi$ is a non-trivial congruence on $\mathbf A$ with $\Phi\subseteq\Theta$. Since $\Theta$ is an atom of $\BCon\mathbf A$ we have $\Theta=\Phi$. Hence $\mathbf A$ has the one-block-property.
\end{proof}

We can prove the following characterization of algebras having the one-block-property.

\begin{theorem}
Let $\mathbf A=(A,F)$ be an algebra. Then the following are equivalent:
\begin{enumerate}[{\rm(i)}]
\item $\mathbf A$ has the one-block-property.
\item If $a,b\in A$ and $(a,b)\in\Theta(x,y)$ for all $(x,y)\in\Theta(a,b)$ with $x\neq y$ then $x,y\in[a]\Theta(a,b)$ for all $(x,y)\in\Theta(a,b)$ with $x\neq y$.
\end{enumerate}
\end{theorem}

\begin{proof}
$\text{}$ \\
(i) $\Rightarrow$ (ii): \\
Assume $a,b\in A$, $(a,b)\in\Theta(x,y)$ for all $(x,y)\in\Theta(a,b)$ with $x\neq y,$ and $(c,d)\in\Theta(a,b)$ with $c\neq d$. Then $a\neq b$. Let $\Phi$ be a non-trivial congruence on $\mathbf A$ with $\Phi\subseteq\Theta(a,b)$. Then there exists some $(e,f)\in\Phi$ with $e\neq f$. We have $\Theta(e,f)\subseteq\Phi\subseteq\Theta(a,b)$. Because of $(e,f)\in\Theta(a,b)$ and $e\neq f$ we conclude $(a,b)\in\Theta(e,f)$ and hence $\Theta(a,b)\subseteq\Theta(e,f)$. Together we obtain $\Theta(e,f)=\Theta(a,b)$ and therefore $\Phi=\Theta(a,b)$ showing that $\Theta(a,b)$ is an atom of $\BCon\mathbf A$. According to (i), $[a]\Theta(a,b)$ is the unique class of $\Theta(a,b)$ which is not a singleton and hence $c,d\in[a]\Theta(a,b)$. \\
(ii) $\Rightarrow$ (i): \\
Assume $\Theta$ to be an atom of $\BCon\mathbf A$. Then there exist $a,b\in A$ with $a\neq b$ and $\Theta(a,b)=\Theta$. Suppose $(c,d)\in\Theta(a,b)$ and $c\neq d$. Then $\theta(c,d)$ is a non-trivial congruence on $\mathbf A$ with $\Theta(c,d)\subseteq\Theta(a,b)$. Since $\Theta(a,b)$ is an atom of $\BCon\mathbf A$ we have $\Theta(c,d)=\Theta(a,b)$ and hence $(a,b)\in\Theta(c,d)$. Because of (ii) we have $x,y\in[a]\Theta(a,b)$ for all $(x,y)\in\Theta(a,b)$ with $x\neq y$. This shows that $[a]\Theta(a,b)$ is the unique class of $\Theta$ which is not a singleton.
\end{proof}

In the following let $P_1(\mathbf A)$ denote the set of unary polynomial functions on the algebra $\mathbf A$.

In the proof of the next lemma we use the well-known fact that an equivalence relation $\Theta$ on the universe of an algebra $\mathbf A$ is a congruence on $\mathbf A$ if and only if $(x,y)\in\Theta$ implies $\big(p(x),p(y)\big)\in\Theta$ for all $p\in P_1(\mathbf A)$.

The proof of the following lemma is straightforward.

\begin{lemma}\label{lem1}
Let $\mathbf A=(A,F)$ be an algebra and $B\subseteq A$. Then the following are equivalent:
\begin{enumerate}[{\rm(i)}]
\item $B^2\cup\omega_A\in\Con\mathbf A$.
\item If $(a,b)\in B^2$ and $p\in P_1(\mathbf A)$ then $\big(p(a),p(b)\big)\in B^2\cup\omega_A$.
\end{enumerate}
\end{lemma}

Next we describe varieties $\mathcal V$ such that for any $\mathbf A=(A,F)\in\mathcal V$ and any $\Theta\in\Con\mathbf A$ there exists some class $B$ of $\Theta$ being a subuniverse of $\mathbf A$ satisfying $B^2\cup\omega_A\in\Con\mathbf A$.

A {\em variety with an absorbing element $1$} is a variety with a unique nullary operation $1$ satisfying the identities
\[
f(x_1,\ldots,x_{i-1},1,x_{i+1},\ldots,x_n)\approx1
\]
for all $n$-ary fundamental operations $f$ and for all $i=1,\ldots,n$. It can be easily shown by induction of term complexity that then also the identities
\[
t(x_1,\ldots,x_{i-1},1,x_{i+1},\ldots,x_n)\approx1
\]
hold for all $n$-ary terms $t$ and for all $i=1,\ldots,n$. It is easy to see that for an algebra $\mathbf A$ the following are equivalent:
\begin{enumerate}[(i)]
\item $\mathbf A$ has an absorbing element $1$.
\item $1$ is the unique nullary operation of $\mathbf A$, and each $p\in P_1(\mathbf A)$ is either a constant function or $p(1)=1$.
\end{enumerate}

\begin{theorem}\label{th3}
Let $\mathcal V$ be a variety with an absorbing element $1$. Then for every $\mathbf A=(A,F)\in\mathcal V$ and every $\Theta\in\Con\mathbf A$, $[1]\Theta$ is the unique class of $\Theta$ being a subuniverse of $\mathbf A$ and, moreover, $([1]\Theta)^2\cup\omega_A\in\Con\mathbf A$.
\end{theorem}

\begin{proof}
Let $\mathbf A=(A,F)\in\mathcal V$ and $\Theta\in\Con\mathbf A$ and put $B:=[1]\Theta$. Since $\mathbf A$ satisfies the identity $f(1,\ldots,1)\approx1$ for every fundamental operation $f$, $B$ is a subuniverse of $\mathbf A$. Because $1$ is a nullary fundamental operation, no other class of $\Theta$ is a subuniverse of $\mathbf A$. Now let $(a,b)\in B^2$ and $p\in P_1(\mathbf A)$. Then $p$ either is a constant function or $p(1)=1$. If $p$ is a constant function then $\big(p(a),p(b)\big)\in\omega_A\subseteq B^2\cup \omega_A$. Otherwise, $\big(p(a),p(b)\big)\in\big([p(1)]\Theta\big)^2=([1]\Theta)^2=B^2\subseteq B^2\cup\omega_A$. According to Lemma~\ref{lem1}, $B^2\cup\omega_A\in\Con\mathbf A$.
\end{proof}

\begin{example}
Examples of varieties with an absorbing element $1$ are join-semilattices with top element $1$, join-directoids with a top element $1$ and semigroups with an absorbing element, usually denoted by $0$.
\end{example}

\begin{corollary}
Let $\mathcal V$ be a variety with an absorbing element $1$. If $\mathbf A\in\mathcal V$ and $|[1]\Theta|>1$ for each non-trivial congruence $\Theta$ on $\mathbf A$ then $\mathbf A$ is a quasi-Rees algebra.
\end{corollary}

\begin{example}
The join-semilattice $\mathbf A=(A,\vee,1)$ with top element $1$ depicted in Figure~1 has the absorbing element $1$.

\vspace*{-3mm}

\begin{center}
\setlength{\unitlength}{7mm}
\begin{picture}(4,4)
\put(2,1){\circle*{.3}}
\put(1,2){\circle*{.3}}
\put(3,2){\circle*{.3}}
\put(2,3){\circle*{.3}}
\put(2,1){\line(-1,1)1}
\put(2,1){\line(1,1)1}
\put(2,3){\line(-1,-1)1}
\put(2,3){\line(1,-1)1}
\put(1.85,.25){$0$}
\put(.3,1.85){$a$}
\put(3.4,1.85){$b$}
\put(1.9,3.4){$1$}
\put(1.25,-.75){{\rm Fig.~1}}
\end{picture}
\end{center}

\vspace*{3mm}

The congruence lattice of $\mathbf A$ is visualized in Figure~2:

\vspace*{-3mm}

\begin{center}
\setlength{\unitlength}{7mm}
\begin{picture}(6,6)
\put(3,1){\circle*{.3}}
\put(2,2){\circle*{.3}}
\put(4,2){\circle*{.3}}
\put(1,3){\circle*{.3}}
\put(3,3){\circle*{.3}}
\put(5,3){\circle*{.3}}
\put(3,5){\circle*{.3}}
\put(3,1){\line(-1,1)2}
\put(3,1){\line(1,1)2}
\put(3,3){\line(-1,-1)1}
\put(3,3){\line(0,1)2}
\put(3,3){\line(1,-1)1}
\put(3,5){\line(-1,-1)2}
\put(3,5){\line(1,-1)2}
\put(2.65,.25){$\omega_A$}
\put(1.1,1.85){$\Theta_1$}
\put(.1,2.85){$\Theta_3$}
\put(5.35,2.85){$\Theta_5$}
\put(3.35,2.85){$\Theta_4$}
\put(4.35,1.85){$\Theta_2$}
\put(2.65,5.4){$A^2$}
\put(2.25,-.75){{\rm Fig.~2}}
\end{picture}
\end{center}

\vspace*{3mm}

with
\begin{align*}
\Theta_1 & :=\{0\}^2\cup\{a\}^2\cup\{b,1\}^2, \\
\Theta_2 & :=\{0\}^2\cup\{b\}^2\cup\{a,1\}^2, \\
\Theta_3 & :=\{0,a\}^2\cup\{b,1\}^2, \\
\Theta_4 & :=\{0\}^2\cup\{a,b,1\}^2, \\
\Theta_5 & :=\{0,b\}^2\cup\{a,1\}^2.
\end{align*}
It is easy to see that $\mathbf A$ is a quasi-Rees algebra having the one-block-property. In accordance with the remark after Definition~\ref{def1}, $\BCon\mathbf A$ is semimodular, but not modular as can easily be checked.
\end{example}

Next we describe Rees congruences on a quotient algebra $\mathbf A/\Theta$ of some algebra $\mathbf A=(A,F)$ with respect to a congruence $\Theta$ on $\mathbf A$. Keep in mind that then a subset of $A/\Theta$ is in fact a set of congruence classes, i.e.\ of subsets of $A$.

\begin{theorem}\label{th4}
Let $\mathbf A$ be an algebra, $\Theta\in\Con\mathbf A$ and $B\subseteq A/\Theta$ and define $C:=\bigcup\limits_{X\in B}X$. Then the following hold:
\begin{enumerate}[{\rm(i)}]
\item The relation $B^2\cup\omega_{A/\Theta}$ on $A/\Theta$ is a congruence on $\mathbf A/\Theta$ if and only if $C^2\cup\Theta\in\Con\mathbf A$. In this case we have
\[
B^2\cup\omega_{A/\Theta}=(C^2\cup\Theta)/\Theta.
\]
\item The subset $B$ of $A/\Theta$ is a subuniverse of $\mathbf A/\Theta$ if and only if $C$ is a subuniverse of $\mathbf A$.
\end{enumerate}
\end{theorem}

\begin{proof}
\
\begin{enumerate}[(i)]
\item Let $p\colon A\rightarrow A$. Then $p\in P_1(\mathbf A)$ if and only if $p$ is either a constant function or there exists some positive integer $m$, some $m$-ary term $t$, some $i\in\{1,\ldots,m\}$ and some $a_1,\ldots,a_{i-1},a_{i+1},\ldots,a_m\in A$ such that
\[
t(a_1,\ldots,a_{i-1},x,a_{i+1},\ldots,a_m)=p(x)
\]
for all $x\in A$. Now let $q:A/\Theta\rightarrow A/\Theta$. Then $q\in P_1(\mathbf A/\Theta)$ if and only if $q$ is either a constant function or there exists some positive integer $m$, some $m$-ary term $u$, some $i\in\{1,\ldots,m\}$ and some $b_1,\ldots,b_{i-1},b_{i+1},\ldots,b_m\in A$ such that
\[
u([b_1]\Theta,\ldots,[b_{i-1}]\Theta,[x]\Theta,[b_{i+1}]\Theta,\ldots,[b_m]\Theta)=q([x]\Theta)
\]
for all $x\in A$. But
\[
u([b_1]\Theta,\ldots,[b_{i-1}]\Theta,[x]\Theta,[b_{i+1}]\Theta,\ldots,[b_m]\Theta)=[u(b_1,\ldots,b_{i-1},x,b_{i+1},\ldots,b_m)]\Theta.
\]
This shows
\begin{enumerate}
\item[(*)] $P_1(\mathbf A/\Theta)=\{[x]\Theta\mapsto[p(x)]\Theta\mid p\in P_1(\mathbf A)\}$.
\end{enumerate}
Now the following are equivalent:
\begin{align*}
& B^2\cup\omega_{A/\Theta}\in\Con(\mathbf A/\Theta), \\
& \text{if }x,y\in A, ([x]\Theta,[y]\Theta)\in B^2\cup\omega_{A/\Theta}\text{ and }p\in P_1(\mathbf A/\Theta)\text{ then} \\
& \hspace*{1cm} \big(p([x]\Theta),p([y]\Theta)\big)\in B^2\cup\omega_{A/\Theta}, \\
& \text{if }x,y\in A, ([x]\Theta,[y]\Theta)\in B^2\text{ and }q\in P_1(\mathbf A)\text{ then }\big([q(x)]\Theta),[q(y)]\Theta\big)\in B^2\cup\omega_{A/\Theta}, \\
& \text{if }(x,y)\in C^2\text{ and }q\in P_1(\mathbf A)\text{ then }\big(q(x),q(y)\big)\in C^2\cup\Theta, \\
& \text{if }(x,y)\in C^2\cup\Theta\text{ and }q\in P_1(\mathbf A)\text{ then }\big(q(x),q(y)\big)\in C^2\cup\Theta, \\
& C^2\cup\Theta\in\Con\mathbf A.
\end{align*}
Note that for proving the equivalence of the second and the third statement we use (*). Finally, if $B^2\cup\omega_{A/\Theta}\in\Con(\mathbf A/\Theta)$ then for all $a,b\in A$ the following are equivalent:
\begin{align*}
& ([a]\Theta,[b]\Theta)\in B^2\cup\omega_{A/\Theta}, \\
& ([a]\Theta,[b]\Theta)\in B^2\text{ or }([a]\Theta,[b]\Theta)\in\omega_{A/\Theta}, \\
& (a,b)\in C^2\text{ or }(a,b)\in\Theta, \\
& (a,b)\in C^2\cup\Theta, \\
& ([a]\Theta,[b]\Theta)\in(C^2\cup\Theta)/\Theta.
\end{align*}
\item For every $m$-ary fundamental operation and for every $a_1,\ldots,a_m\in C$ the following are equivalent: $f([a_1]\Theta,\ldots,[a_m]\Theta)\in B$; $[f(a_1,\ldots,a_m)]\in B$; $f(a_1,\ldots,a_m)\in C$.
\end{enumerate}
\end{proof}

Using Theorem~\ref{th4} we can characterize quasi-Rees quotient algebras as follows.

\begin{corollary}
Let $\mathbf A=(A,F)$ be an algebra and $\Theta\in\Con\mathbf A$. Then the following are equivalent:
\begin{enumerate}[{\rm(i)}]
\item $\mathbf A/\Theta$ is a quasi-Rees algebra.
\item Every congruence on $\mathbf A$ strictly including $\Theta$ has some class $C$ being a subuniverse of $\mathbf A$, but not a class of $\Theta$, such that $C^2\cup\Theta\in\Con\mathbf A$.
\end{enumerate}
\end{corollary}

\begin{proof}
Consider the following statement:
\begin{enumerate}
\item[(iii)] Every non-trivial congruence on $\mathbf A/\Theta$ has some class $B$ being a non-trivial subuniverse of $\mathbf A/\Theta$ such that $B^2\cup\omega_{A/\Theta}\in\Con(\mathbf A/\Theta)$.
\end{enumerate}
(i) $\Leftrightarrow$ (iii): \\
This follows from the definition of a quasi-Rees algebra. \\
(iii) $\Rightarrow$ (ii): \\
Put $C:=\bigcup\limits_{X\in B}X$ and use Theorem~\ref{th4} (i). \\
(ii) $\Rightarrow$ (iii): \\
Put $B:=\{[x]\Theta\mid x\in C\}$ and use Theorem~\ref{th4} (i).
\end{proof}

Similarly, we can characterize quotient algebras having the one-block-property as follows.

\begin{theorem}
Let $\mathbf A=(A,F)$ be some algebra and $\Theta\in\Con\mathbf A$. Then the following are equivalent:
\begin{itemize}
\item $\mathbf A/\Theta$ has the one-block-property.
\item For every congruence $\Phi$ on $\mathbf A$ covering $\Theta$ there exists some subset $B$ of $A/\Theta$ with $(\bigcup\limits_{X\in B}X)^2\cup\Theta=\Phi$.
\end{itemize}
\end{theorem}

\begin{proof}
For every congruence $\Phi$ on $\mathbf A$ including $\Theta$ we have: $\Phi/\Theta$ is an atom of $\BCon\mathbf A$ if and only if $\Phi$ covers $\Theta$. Now for every $a\in A$ and every subset $B$ of $A/\Theta$, $[a]\Theta\in B$ is equivalent to $a\in\bigcup\limits_{X\in B}X$. Hence for every subset $B$ of $A/\Theta$ the following are equivalent:
\begin{align*}
& \Phi/\Theta=B^2\cup\omega_{A/\Theta}, \\
& \text{For all }y,z\in A, ([y]\Theta,[z]\Theta)\in\Phi/\Theta\text{ if and only if }([y]\Theta,[z]\Theta\in B\text{ or }[y]\Theta=[z]\Theta), \\
& \text{For all }y,z\in A, (y,z)\in\Phi\text{ if and only if }\big(y,z\in\bigcup\limits_{X\in B}X\text{ or }(y,z)\in\Theta\big), \\
& \Phi=(\bigcup\limits_{X\in B}X)^2\cup\Theta.
\end{align*}
Thus the following are equivalent:
\begin{align*}
& \mathbf A/\Theta\text{ satisfies the one-block-property}, \\
& \text{For every congruence }\Phi\text{ on }\mathbf A\text{ covering }\Theta\text{ there exists some subset }B\text{ of }A/\Theta\text{ with} \\
& \hspace*{1cm} B^2\cup\omega_{A/\Theta}=\Phi/\Theta, \\
& \text{For every congruence }\Phi\text{ on }\mathbf A\text{ covering }\Theta\text{ there exists some subset }B\text{ of }A/\Theta\text{ with} \\
& \hspace*{1cm}(\bigcup\limits_{X\in B}X)^2\cup\Theta=\Phi.
\end{align*}
\end{proof}

Authors' addresses:

Ivan Chajda \\
Palack\'y University Olomouc \\
Faculty of Science \\
Department of Algebra and Geometry \\
17.\ listopadu 12 \\
771 46 Olomouc \\
Czech Republic \\
ivan.chajda@upol.cz

Helmut L\"anger \\
TU Wien \\
Faculty of Mathematics and Geoinformation \\
Institute of Discrete Mathematics and Geometry \\
Wiedner Hauptstra\ss e 8-10 \\
1040 Vienna \\
Austria, and \\
Palack\'y University Olomouc \\
Faculty of Science \\
Department of Algebra and Geometry \\
17.\ listopadu 12 \\
771 46 Olomouc \\
Czech Republic \\
helmut.laenger@tuwien.ac.at
\end{document}